\documentclass[11pt,leqno]{amsart}
\usepackage{amsmath,amssymb,amsthm}
\newtheorem{theorem}{Theorem}[section]
\newtheorem{lemma}[theorem]{Lemma}

\newtheorem{corollary}[theorem]{Corollary}
\theoremstyle{definition}

\theoremstyle{remark}

\numberwithin{equation}{section}

\def\fnote#1{\footnote}

\def\ignora#1{}
\def\n3#1{\left\vert  \! \left\vert \! \left\vert \, #1 \, \right\vert \!
  \right\vert \! \right\vert }

\begin{document}

\title{ Subspaces of Banach spaces with big slices }

\author{Julio Becerra Guerrero, Gin{\'e}s L{\'o}pez-P{\'e}rez and Abraham Rueda Zoca}
\address{Universidad de Granada, Facultad de Ciencias.
Departamento de An\'{a}lisis Matem\'{a}tico, 18071-Granada
(Spain)} \email{juliobg@ugr.es, glopezp@ugr.es, arz0001@correo.ugr.es}

\begin{abstract} We study when diameter two properties  pass down to subspaces.  We obtain that the slice two property (respectively diameter two property, strong diameter two property) passes down from a Banach space  $X$ to a subspace $Y$ whenever $Y$ is complemented by a norm one projection with finite-dimensional kernel (respectively the quotient $X/Y$
is finite dimensional, $X/Y$ is strongly regular). Also we study the same problem for dual properties of the above ones, as having octahedral, weakly octahedral or 2-rough norm.

\end{abstract}

\maketitle\markboth{J. Becerra, G. L\'{o}pez and A. Rueda}{Subspaces of spaces with big slices.}

\section{Introduction.} We recall that a Banach space $X$ satisfies the strong diameter two property (SD2P), respectively diameter two property (D2P), slice-diameter two property (slice-D2P),  if every convex combination of slices, respectively every nonempty relatively weakly open subset , every slice,  in the unit ball of $X$ has diameter $2$. The  weak-star slice diameter two property      ($w^*$-slice-D2P), weak-star diameter property ($w^*$-D2P) and weak-star strong diameter two property ($w^*$-SD2P) for a dual Banach spaces are defined as usual, changing slices by $w^*$-slices and weak open subsets by $w^*$- open subsets in the unit ball. It is known that Daugavet property implies the SD2P and so the D2P and slice-D2P.  The above connection between the Daugavet properties and the diameter two properties was discovered in \cite{shv}. In fact, the dual of a Banach space with Dagavet property also satisfy the $w^*$-SD2P \cite{blr}. It is known that the above six properties are extremely different as it is proved in \cite{avd}.

As can be easily seen, the above properties are not inhered to subspaces, so  given a Banach space $X$ with some of the above properties, a natural question is wonder what subspaces $Y$ of $X$ have also the those property.  Up to our best knowledge, the same problem has not been yet considered for the diameter two properties. The aim of this note is  to study when the diameter two properties  pass down to subspaces. We show  that the above properties are not three space properties. We obtain  that a subspace $Y$ of a Banach space $X$ with the  SD2P, has the  SD2P whenever the quotient $X/Y$ is strongly regular, a weaker property than RNP. In particular, it is showed that  the above holds if $X/Y$ does not contain $\ell_1^n$ uniformly.  For the D2P, we get that a subspace $Y$ of a Banach space $X$ with D2P, has D2P whenever $X/Y$ is finite-dimensional and also we obtain that  a subspace $Y$ of a Banach space $X$ with slice-D2P, has slice-D2P if $Y$ is cofinite-dimensional and one complemented.  

Finally, we study the same problem for the dual properties of diameter two properties, like having octahedral, weakly octahedral or $2-$rough norm. Recall that a Banach space $X$ satisfies the SD2P if, and only if, the norm of $X^*$ is octahedral and the norm of $X$ is octahedral if, and only if, $X^*$ has $w^*$-SD2P (see \cite{blr}).

The norm of a Banach space $X$ is octahedral if for every $\varepsilon>0$ and for every finite-dimensional subspace $M$ of $X$ there is some $y$ in the unit sphere of $X$ such that $$\Vert x+\lambda y\Vert\geq (1-\varepsilon)(\Vert x\Vert +\vert \lambda\vert)$$ holds for every $x\in M$ and for every scalar $\lambda$ (see \cite{dgz}).

 Similarly, following \cite{hlp} and  \cite[Proposition I.1.11]{dgz}, a Banach space $X$ satisfies  D2P, respectively slice-D2P, if the norm of the dual space $X^*$ is weakly octahedral, respectively $2$-rough. Also,  the norm of $X$ is weakly octahedral, respectively  $2-$rough, if, and only if, $X^*$ has $w^*$-D2P, respectively $w^*$-slice-D2P. 

The norm of a Banach space $X$ is weakly octahedral (see \cite{hlp}) if for every
finite-dimensional subspace $Y$ of $X$, every $x^*\in B_{X^*}$,
and every $\varepsilon\in\mathbb R^+$ there exists $y\in S_X$
satisfying

$$\Vert x+y\Vert\geq (1-\varepsilon)(\vert x^*(x)\vert+\Vert y\Vert)\ \forall x\in Y.$$
The norm is said to be $2-$rough if, for every $u$ in the unit sphere of $X$ one has
$$\lim\sup_{\Vert h\Vert \to 0}\frac{\Vert u+h\Vert +\Vert u-h\Vert -2}{\Vert h\Vert}=2$$

Recall that a closed, bounded and convex subset $C$ of a  Banach space $X$ is said to be strongly regular if every closed and convex subset of $C$ has convex combinations of slices with diameter arbitrarily small, equivalently convex combinations of relatively weakly open subsets with diameter arbitrarily small, since it is known that every nonempty relatively weakly open subset contains a convex combination of slices \cite[Lemma 5.3]{bourgain}. We refer to this fact like the Bourgain lemma. A point $x$ of a such set $C$ of $X$ is said to be a point of strong regularity if there are convex combinations of slices in $C$ containing $x$ with arbitrarily small diameter. If $C$ is strongly regular then $C$ contains a norm dense subset of points of strong regularity \cite[Proposition 3.6]{ggms}. The strong regularity is a strictly weaker property than RNP and it is known that, for a Banach space $X$, $X^*$ is strongly regular if, and only if, $X$ does not contain isomorphic copies of $\ell_1$ \cite[Corollary 6.18]{ggms}. Also it is known that Banach spaces not containing $\ell_1^n$ uniformly are strongly regular \cite[Proposition 2.14]{chacho}. Banach spaces not containing $\ell_1^n$ uniformly are called too $K$-convex Banach spaces.

We consider real Banach spaces, $B_X$ (resp. $S_X$) denotes the closed unit ball (resp. sphere) of the Banach space $X$. All subspaces of a Banach space will be considered closed subspaces. If $Y$ is a subspace of a Banach space $X$, $X^*$ stands for the dual space of $X$ and the annihilator of $Y$ is the subspace of $X^*$ given by 
$$Y^{\circ}=\{x^*\in X^*:\ x^*(Y)=0\}.$$
A slice of a bounded subset $C$ of $X$ is the set $$S(C,f,\alpha):=\{x\in C:\ f( x)>M-\alpha\},$$ where $f\in X^*$, $f\neq 0$, $M=\sup_{x\in C}f(x)$ and $\alpha>0$. If $X=Y^*$ is a dual space for some Banach space $Y$ and $C$ is a bounded subset of $X$, a $w^*$-slice of $C$ is the set $$S(C,y,\alpha):=\{f\in C:f(y)>M-\alpha\},$$ where  $y\in Y$, $y\neq 0$, $M=\sup_{f\in C}f(y)$ and $\alpha>0$. $w$ (resp. $w^*$) denotes the weak (resp. weak-star) topology of a Banach space.

\bigskip

\section{Main results.}

\bigskip

 We begin with the following question: can a closed subspace satisfying any diameter two property force the space to have any diameter two property? We will see that, in general, this is false.

Indeed, it is possible construct an example $Z=X\times Y$ for $X=Y=C([0,1])$
and a norm $\Vert\cdot\Vert$ in $Z$ such that $Z$ fails 
 every diameter two property in spite of the fact that $X$ and $Y$ has the   strong diameter two property with the norm of $Z$ restricted to $X$ and $Y$
respectively. In order to show this norm we need the following

\begin{lemma}\label{dentings} Let $X$ be a Banach space and assume that $C=\overline{co}(B\cup\{x_0\})$
for some closed bounded and convex subset $B$ of $X$ and $x_0\in
X\setminus B$. Then $x_0$ is a denting point of $C$. Moreover, if 
$B$ is the unit ball of $X$, then $\overline{co}(C\cup\{-x_0\})$ is the unit ball
of some equivalent norm in $X$ and $x_0$ is a denting point of the
unit ball   for this new norm.\end{lemma}

\begin{proof} As $x_0\notin B$ we can find, by a separation
argument, $x^*\in S_{X^*}$ such that $x^*(x_0)>\sup x^*(B)=M$.
Hence $\sup\ x^*(C)= x^*(x_0)$. Fix $\varepsilon >0$, let $\beta:=\sup\limits_{x\in B}\Vert x\Vert$ and $0<\alpha
<\frac{x^*(x_0)-M}{2 (\beta+\Vert x_0\Vert)}\varepsilon$.
Consider $S=\{x\in C: x^*(x)>x^*(x_0)-\alpha\}$. Now $S$ is a
slice of $C$. Pick $y,\ z\in co(B\cup\{x_0\})\cap S$ with
$y=\lambda b_1+(1-\lambda)x_0$ and $z=\mu b_2+(1-\mu)x_0$, for
some $0\leq \lambda,\ \mu\leq 1$ and $b_1,\ b_2\in B$. As $y,\
z\in S$ we deduce that $\lambda,\ \mu<
\frac{\varepsilon}{2(\beta +\Vert x_0\Vert)}$. Hence $\Vert y-z\Vert<
\frac{\varepsilon}{2(\beta+\Vert x_0\Vert)}(\Vert b_1\Vert +\Vert b_2\Vert +\Vert x_0\Vert)<\varepsilon$. This proves that $co(B\cup\{x_0\})\cap S$
has diameter less than $\varepsilon$ and so $S$ is a slice of $C$
containing $x_0$ with diameter less than $\varepsilon$ and $x_0$
is a denting point of $C$.

In the case that $B$ is in particular the unit ball of $X$, then it is
easy to see that the above set $S$ is a slice of $\overline{co}(C\cup\{-x_0\})$
containing $x_0$ with diameter less than $\varepsilon$ for
$\alpha$ enough  small.\end{proof}

In order to exhibit the announced example, let $B$ be the closed unit ball of
 $C([0,1])\oplus_{\infty}C([0,1])$ and
 $C=\overline{co}(B\cup\{(x_0,0)\}\cup\{(-x_0,0)\})$, where $x_0$ is a
 point in $C([0,1])$ whose usual norm in $C([0,1])$ is $2$. From lemma
 \ref{dentings}, $C$ is the unit ball of some norm in $C([0,1])\times
 C([0,1])$ failing every diameter two property whose restriction to the
 factor spaces has the SD2P.

Recall that a property ($\mathcal P$) is said to be a tree space property if a Banach space $X$ satisfies ($\mathcal P$) whenever there exists $Y\subseteq X$ a closed subspace such that $Y$ and $X/Y$ enjoy the property ($\mathcal P$). 
  
 As a consequence of the previous lemma, the   diameter two properties   are not tree space properties.

We study now the following question: given a Banach space $X$ satisfying some diameter two property, which closed subspaces of $X$ enjoy any diameter two property? The following result answer the above question assuming natural properties to the quotient $X/Y$.

\begin{theorem}\label{teocociented2p}

Let $X$ be a Banach space and let $Y$ a  subspace of $X$.

\begin{enumerate}

\item[i)] If $X$ has the slice diameter two property  and there exists $\pi:X\longrightarrow Y$ a norm-one  linear and continuous projection onto a subspace $Y$ of $X$ such that $ker(\pi)$ is finite-dimensional, then $Y$ has the slice diameter two property. 

\item[ii)] If $X$ has the diameter two property and $X/Y$ is finite-dimensional, then $Y$ has the diameter two property.

\item[iii)] If $X$ has the strong diameter two property and $X/Y$ is strongly regular then $Y$ has the strong diameter two property.

 \end{enumerate}

\end{theorem}

\begin{proof}

{\bf i)} Note  that, since $\Vert \pi\Vert =1$, 

$$\pi(B_X)=B_Y.$$

Under the hypotheses of i) is proved in  \cite[Theorem 5.3]{scsewe}  that

$$ 2=\inf \{ diam(S)\ /\ S\mbox{ slice of  }B_X \}\leq \inf \{diam(T)\ /\ T\mbox{ slice of } \pi(B_X)=B_Y\}. $$

Thus $Y$ has the slice diameter two property, as desired.

{\bf ii)} Consider

$$W:=\{y\in Y\ /\ \vert y_i^*(y-y_0)\vert<\varepsilon_i\ \forall i\in\{1,\ldots, n\}\}, $$

for $n\in\mathbb N, \varepsilon_i\in\mathbb R^+, y_i^*\in Y^*$ for each $i\in\{1,\ldots,n\}$ and $y_0\in Y$ such that

$$W\cap B_Y\neq \emptyset.$$

Let's prove that $W\cap B_Y$ has diameter 2. To this aim pick an arbitrary  $\delta\in\mathbb R^+$.

Assume that $y_i^*\in X^*$ for each  $i\in\{1,\ldots, n\}$. We have no loss of generality by Hahn-Banach's theorem.

Define

$$U:=\{x\in X\ /\ \vert y_i^*(x-y_0)\vert<\varepsilon_i\ \forall i\in\{1,\ldots, n\}\},$$

which is a weakly open set in $X$ such that $U\cap B_X\neq \emptyset$.

Let $p:X\longrightarrow X/Y$ be the quotient map, which is a $w-w$ open map. Then  $p(U)$ is a weakly open set in  $X/Y$. In addition

$$\emptyset\neq p(U\cap B_X)\subseteq p(U)\cap p(B_X)=p(U)\cap B_{X/Y}.$$

Defining $A:=p(U)\cap B_{X/Y}$, then $A$ is a non-empty ,relatively weakly open and convex subset of $B_{X/Y}$ which contains to zero.   Hence, as $X/Y$ is finite-dimensional, we can find a weakly open set $V$ of $X/Y$, in fact a ball centered at $0$, such that $V\subset A$ and that

\begin{equation}\label{diamabicocie}
diam(V\cap p(U)\cap B_{X/Y})= diam(V)<\frac{\delta}{8}.
\end{equation}

As $V\subset A$ then $B:=p^{-1}(V)\cap U\cap B_X\neq \emptyset$. Hence $B$ is a non-empty relatively weakly open subset of $B_X$. Using that $X$ satisfies the diameter two property we can assure the existence of  $x,y\in B$ such that

\begin{equation}\label{estimabiespad2p}
\Vert x-y\Vert>2-\frac{\delta}{8}.
\end{equation}

Note that  $x\in B$ implies $p(x)\in V= V\cap P(U)\cap B_{X/Y}$. In view of  (\ref{diamabicocie}) it follows

$$\Vert p(x)\Vert\leq  diam(V\cap p(U)\cap B_{X/Y})<\frac{\delta}{8}.$$

Hence there exists $u\in B_Y$ such that

\begin{equation}\label{elemenposd2p}
\Vert x-u\Vert<\frac{\delta}{8}.
\end{equation}

Again using (\ref{diamabicocie}), by a similar argument we can find $v\in B_Y$ satisfying

\begin{equation}\label{elemenegd2p}
\Vert v-y\Vert<\frac{\delta}{8}.
\end{equation}

Note that given $i\in\{1,\ldots, n\}$, keeping in mind (\ref{elemenposd2p}), one has

$$\vert y_i^*(u-y_0)\vert\leq \vert y_i^*(y-x)\vert+\vert y_i^*(x-y_0)\vert\leq \Vert y_i^*\Vert\frac{\delta}{8}+\varepsilon_i,$$

using that $x\in U$. Thus, if we define

$$W_\delta:=\left\{y\in Y\ /\ \vert y_i^*(y-y_0)\vert<\varepsilon_i+\Vert y_i^*\Vert\frac{\delta}{8}\ \forall i\in\{1,\ldots, n\}\right\}$$

it follows that $u,v\in W_\delta\cap B_Y$. On the other hand, in view of  (\ref{estimabiespad2p}),(\ref{elemenposd2p}) and (\ref{elemenegd2p}) we can estimate

\begin{equation}\label{estimafinald2p}
diam(W_\delta\cap B_Y)\geq \Vert u-v\Vert\geq \Vert x-y\Vert-\Vert x-u\Vert-\Vert y-v\Vert\geq 2-\frac{\delta}{8}-\frac{\delta}{8}-\frac{\delta}{8}>2-\delta.
\end{equation}

As $\delta\in\mathbb R^+$ was arbitrary we deduce that $diam(W\cap B_Y)=2$, as wanted.

{\bf iii)} Assume that $X$ has the strong diameter two property and that $X/Y$ is strongly regular.

Let $C:=\sum_{i=1}^n \lambda_i S(B_Y,y_i^*,\varepsilon)=\sum_{i=1}^n \lambda_i S_i$ a convex combination of slices of $B_Y$. Let prove that $diam(C)=2$. To this aim pick an arbitrary $\delta\in\mathbb R^+$.

Let $\pi:X\longrightarrow X/Y$ the quotient map. Assume that $y_i^*\in S_{X^*}$. We have no loss of generality because of Hahn-Banach's theorem. 

For each $i\in\{1,\ldots, n\}$ consider $A_i:=\pi(S(B_X,y_i^*,\varepsilon))$, which is a convex subset of $B_{X/Y}$ containing to zero. By \cite[Proposition III.6]{ggms} then $\overline{A_i}$ is equal to the closure of the set of its strongly regular  points. As a consequence, for each  $i\in\{1,\ldots, n\}$, there exists $a_i$ a strongly regular point of $\overline{A_i}$ such that

\begin{equation}\label{normafuertexp}
\left\Vert a_i\right\Vert<\frac{\delta}{16}.
\end{equation}

For every $i\in\{1,\ldots,n\}$ we can find $n_i\in\mathbb N, \mu_1^i,\ldots, \mu_{n_i}^i\in ]0,1]$ such that $\sum_{j=1}^{n_i} \mu_j^i=1$ and $(a_1^i)^*,\ldots, (a_{n_i}^i)^*\in S_{(X/Y)^*},\eta_j^i\in\mathbb R^+$ satisfying that 
$$\sum_{j=1}^{n_i}\mu_j^i (S(B_{X/Y},(a_j^i)^*,\eta_j^i)\cap \overline{A_i})$$ and also satisfying 

\begin{equation}\label{diaslicescociente}
diam\left(\sum_{j=1}^{n_i} \mu_j^i (S(B_{X/Y},(a_j^i)^*,\eta_j^i)\cap \overline{A_i})\right)<\frac{\delta}{16}.
\end{equation}

It is clear that for $i\in\{1,\ldots, n\}$ and $j\in\{1,\ldots, n_i\}$ one has 

$$S(B_{X/Y},(a_j^i)^*,\eta_j^i)\cap A_i\neq \emptyset\Rightarrow S(B_X,\pi^*((a_j^i)^*),\eta_j^i)\cap S(B_X,y_i^*,\varepsilon)\neq \emptyset.$$

Now $\sum_{i=1}^n \lambda_i \sum_{j=1}^{n_i}\mu_j^i (S(B_X,\pi^*((a_j^i)^*),\eta_i)\cap S(B_X,y_i^*,\varepsilon))$ is a convex combination of non-empty relatively weakly open subsets of $B_X$. By a known lemma of Bourgain, last set contains a convex combination of slices of $B_X$ and, as a consequence, last set has diameter two. Hence we can find, for each $i\in\{1,\ldots, n\}, j\in\{1,\ldots, n_i\}$, elements $x_j^i,z_j^i\in S(B_X,\pi^*((a_j^i)^*),\eta_i)\cap S(B_X,y_i^*,\varepsilon)$ verifying

\begin{equation}\label{estinormaespacio}
\left\Vert \sum_{i=1}^n \lambda_i \sum_{j=1}^{n_i}\mu_j^i x_j^i-\sum_{i=1}^n \lambda_i \sum_{j=1}^{n_i}\mu_j^i z_j^i\right\Vert>2-\frac{\delta}{16}.
\end{equation}

By the one hand, given $i\in\{1,\ldots,n\}$ one has

$$\sum_{j=1}^{n_i} \mu_j^i x_j^i\in \sum_{j=1}^{n_i} \mu_j^i S(B_X,\pi^*((a_j^i)^*),\eta_j^i)\cap S(B_X,y_i^*,\varepsilon)\Rightarrow$$

$$\Rightarrow  \pi\left(\sum_{j=1}^{n_i}\mu_j^i x_j^i\right)\in \sum_{j=1}^{n_i}\mu_j^i S(S_{X/Y},(a_j^i)^*,\eta_j^i)\cap A_i,$$

thus, since (\ref{diaslicescociente}), we have

\begin{equation}\label{normaelemento}
\left\Vert \pi\left(\sum_{j=1}^{n_i}\mu_j^i x_j^i\right)\right\Vert\leq \left\Vert \sum_{j=1}^{n_i}\mu_j^i  a_i\right\Vert+
\end{equation}
\begin{equation}
diam\left(\sum_{j=1}^{n_i} \mu_j^i (S(S_{X/Y},(a_j^i)^*,\eta_j^i)
\cap A_i)\right)<\frac{\delta}{8}.
\end{equation}

Hence, for each $i\in\{1,\ldots, n\}$, there exists $a_i\in B_Y$ such that

\begin{equation}\label{elementopositivo}
\left\Vert a_i-\sum_{j=1}^{n_i}\mu_j^ix_j^i\right
\Vert<\frac{\delta}{8}.
\end{equation}

By a similar argument we can find, for every $i\in\{1,\ldots, n\}$, an element  $b_i\in B_Y$ verifying

\begin{equation}\label{elementonegativo}
\left\Vert b_i-\sum_{j=1}^{n_i}\mu_j^iz_j^i\right
\Vert<\frac{\delta}{8}.
\end{equation}

Thus given $i\in\{1,\ldots, n\}$ we deduce, in view of (\ref{elementopositivo}),

$$y_i^*(a_i)=y_i^*\left(
\sum_{j=1}^{n_i}\mu_j^ix_j^i\right)+
y_i^*\left(a_i
-\sum_{j=1}^{n_i}\mu_j^ix_j^i\right)>
1-\varepsilon-\frac{\delta}{8}$$.

In a similar way, using (\ref{elementonegativo}), we deduce that

$$y_i^*(b_i)>1-\varepsilon-\frac{\delta}{8}.$$

Summarizing,

\begin{equation}\label{nuevoslices}
a_i,b_i\in S\left(B_Y,y_i^*,\varepsilon+\frac{\delta}{8}\right).
\end{equation}

On the other hand, in view of (\ref{estinormaespacio}), we deduce:

$$
diam\left(\sum_{i=1}^n \lambda_i S\left(B_Y,y_i^*,\varepsilon+\frac{\delta}{2}\right)\right)\geq \left\Vert \sum_{i=1}^n \lambda_i a_i-\sum_{i=1}^n \lambda_i b_i\right\Vert\geq$$

$$ \geq \left\Vert \sum_{i=1}^n \lambda_i \sum_{j=1}^{n_i}\mu_j^i x_j^i-\sum_{i=1}^n \lambda_i \sum_{j=1}^{n_i}\mu_j^i z_j^i\right\Vert-
\left\Vert a_i-\sum_{j=1}^{n_i}\mu_j^ix_j^i\right
\Vert-\left\Vert b_i- \sum_{j=1}^{n_i}\mu_j^iz_j^i\right
\Vert>$$ $$2-\delta.$$

As $\delta\in\mathbb R^+$ was arbitrary we conclude that $diam(C)=2$, as desired.\end{proof}

We don't know if $Y$ has slice-D2P (D2P) whenever $X$ has slice-D2P (D2P) and $X/Y$ is reflexive or RNP. We don't know even if the part i) of above theorem holds without assuming the norm one projection condition. Note that in this case $Y$ is finite-codimensional.

\begin{corollary}\label{segundo}  Let $X$ be a Banach space such that $X^*$ has   SD2P  and let $Y$ a subspace of  $X$.
If $Y$ does not contain any copy of $\ell_1$ then $Y^{\circ}$ has   SD2P. \end{corollary}

\begin{proof}
Assume that $Y$ does not contain any isomorphic copy of $\ell_1$. Then $Y^*=X^*/Y^\circ$ is strongly regular \cite{ggms}. By Theorem \ref{teocociented2p}, part iii), we deduce that $Y^\circ=(X/Y)^*$ has the SD2P. \end{proof}

Note that, taking into account the duality between SD2P and octahedrality, we deduce from the  corollary \ref{segundo} that the norm of $(X/Y)^{**}=X^{**}/Y^{\circ \circ}$ is octahedral whenever the norm of $X^{**}$ is octahedral and $Y$ does not contain isomorphic copies of $\ell_1$.

\begin{corollary} Let $X$ be a Banach space  the SD2P and $Y$ a subspace of $X$. If $Y^{\circ}$ does not contain isomorphic copies of $\ell_1$, then $Y$ satisfies the   SD2P.\end{corollary}

\begin{proof} As   $Y^{\circ}$ does not contain isomorphic copies of $\ell_1$, we deduce that $(Y^{\circ})^*$ is strongly regular and so $X/Y$ is also strongly regular as a subspace of  $(Y^{\circ})^*$. Now Theorem \ref{teocociented2p} applies.\end{proof}

Note that , taking into account the duality between SD2P and octahedrality, we deduce from the above corollary that the norm of $Y^*$ is octahedral whenever  the norm of $X^*$ is octahedral and $Y^{\circ}$ does not contain isomorphic copies of $\ell_1$.

 \begin{corollary} Let $X$ be a Banach space with  the SD2P and $Y$ a subspace of $X$. If $X/Y$ does not contain $\ell_1^n$ uniformly, then $Y$ has  the SD2P. \end{corollary}
 
 Note that , taking into account the duality between SD2P and octahedrality, we deduce from the above corollary that the norm of $Y^*$ is octahedral whenever  the norm of $X^*$ is octahedral and $Y^{\circ}$ does not contain $\ell_1^n$ uniformly.

It is known that the previous properties are characterized by the dual version of diameter two property. Indeed, given $X$ a Banach space, it is known that $X$ has octahedral norm if, and only if, $X^*$ has the $w^*$-strong diameter two property \cite[Theorem 2.1]{blr}. Moreover, it is known that $X$ has $2-$rough (respectively weakly octahedral) norm if, and only if, $X^*$ has the $w^*$-slice diameter two property (respectively the $w^*$-diameter two property), see \cite{dgz} and \cite[Theorem 3.1]{hlp}.

Now we will see that given $X$ a Banach space and $Y\subseteq X$ a subspace, if $X$ has octahedral, weakly octahedral or $2-$rough norm and $Y$ satisfies some additional condition, we can ensure that $X/Y$ has octahedral norm, weakly octahedral or $2-$rough, respectively.   

\begin{theorem}
Let $X$ be a Banach space and  $Y$ a  subspace of $X$.

\begin{enumerate}
\item[i)] If $X$ has a $2-$rough norm, $Y$ is a finite-dimensional subspace of $X$  and $\pi :X\rightarrow Y$ a norm one projection, then $X/Y$ has a $2-$rough norm. 

\item[ii)] If $X$ has weakly octahedral norm and  $Y$ is finite-dimensional, then $X/Y$ has weakly octahedral norm.

\item[iii)] If $X$ has octahedral norm and $Y$ is reflexive then $X/Y$ has octahedral norm.
\end{enumerate}
\end{theorem}

\begin{proof}

{\bf i)} From the duality between $2-$roughness and $w^*$-slice-D2P, we know that $X^*$ has $w^*$-slice-D2P and we have to prove  that $Y^{\circ}$ has $w^*$-slice-D2P. If we define $Z=Ker(\pi)$, we get that $Z^{\circ}$ is finite-dimensional, since $Y$ it is.  From the $w^*$ version of part i) in  Theorem \ref{teocociented2p} it is enough to see that there is a norm one projection from $X^*$ onto $Y^{\circ}$ with $Ker(p)=Z^{\circ}$. Consider $i$ the inclusion of $Y$ in $X$. Doing $p=\pi^*\circ i^*$ we obtain the desired projection.

{\bf ii)]} In order to prove that $X/Y$ has weakly octahedral norm, we have to check that $(X/Y)^*\cong Y^\circ$ has the $w^*$-diameter two property.

So consider

$$W:=\{y^*\in Y^\circ\ /\ \vert (y^*-y_0^*)(y_i)\vert<\varepsilon_i\ \forall i\in\{1,\ldots, n\}\}, $$

for $n\in\mathbb N, \varepsilon_i\in\mathbb R^+, y_i\in Y$ for each $i\in\{1,\ldots,n\}$ and $y_0^*\in Y^\circ$ such that

$$W\cap B_{Y^\circ}\neq \emptyset.$$

Let's prove that $W\cap B_{Y^\circ}$ has diameter 2. To this aim pick  $\delta\in\mathbb R^+$.

Define

$$U:=\{x^*\in X^*\ /\ \vert (y^*-y_0^*)(y_i)\vert<\varepsilon_i\ \forall i\in\{1,\ldots, n\}\},$$

which is a non-empty relatively weakly$^*$ open of $X^*$ satisfying that $U\cap B_{X^*}\neq \emptyset$.

Let $p:X^*\longrightarrow X^*/Y^\circ\cong Y^*$ be the quotient map, which is a  $w^*-w^*$ open map. Then $p(U)$ is a weakly$^*$ open set of $X^*/Y^\circ$. Moreover

$$\emptyset\neq p(U\cap B_{X^*})\subseteq p(U)\cap p(B_{X^*})=p(U)\cap B_{X^*/Y^\circ}.$$

If we define $A:=p(U)\cap B_{X^*/Y^\circ}$, then we have that $A$ is a relatively weak-star open and  convex subset of $B_{X^*/Y^\circ}$ which contains to zero.

As $Y$ is finite-dimensional, we can find $V$ a weak-star open set of $X^*/Y^\circ$, in fact a ball centered at zero, such that $V\subset A$ and that

\begin{equation}\label{diamabicocie3}
diam(V)=diam(V\cap p(U)\cap B_{X^*/Y^\circ})<\frac{\delta}{8}.
\end{equation}

As $V\subset A$ then $B:=p^{-1}(V)\cap U\cap B_X\neq \emptyset$. Hence $B$ is a non-empty relatively weak-star open subset of  $B_{X^*}$. Since $X$ has
weakly octahedral norm, we deduce that $B$ has diameter 2. Hence we can find $x^*,y^*\in B$ satisfying

\begin{equation}\label{estimabiespa3}
\Vert x^*-y^*\Vert>2-\frac{\delta}{16}.
\end{equation}

Now $x^*\in B$ implies that $p(x^*)\in V= V\cap P(U)\cap B_{X^*/Y^\circ}$. In view of (\ref{diamabicocie3}) it follows that 

$$\Vert p(x^*)\Vert\leq   diam(V\cap p(U)\cap B_{X^*/Y^\circ})<\frac{\delta}{8}.$$

Hence we can find $u^*\in B_{Y^\circ}$ such that

\begin{equation}\label{elemenpos3}
\Vert x^*-u^*\Vert<\frac{\delta}{8}.
\end{equation}

Keeping in mind (\ref{diamabicocie3}) and using a similar argument we can find $v^*\in B_Y$ such that

\begin{equation}\label{elemeneg3}
\Vert v^*-y^*\Vert<\frac{\delta}{8}.
\end{equation}

Note that given $i\in\{1,\ldots, n\}$, by (\ref{elemenpos3}) one has

$$\vert (u^*-y_0^*)(y_i)\vert\leq \vert (x^*-y_0^*)(y_i)\vert+\vert (x^*-u^*)(y_i)\vert\leq \Vert y_i\Vert\frac{\delta}{8}+\varepsilon_i$$

because $x^*\in U$. Thus, if we define

$$W_\delta:=\left\{y^*\in Y^\circ\ /\ \vert (y^*-y_0^*)(y_i)\vert<\varepsilon_i+\Vert y_i\Vert\frac{\delta}{8}\ \forall i\in\{1,\ldots, n\}\right\},$$

we have that $u^*,v^*\in W_\delta\cap B_{Y^\circ}$. On the other hand, in view of (\ref{estimabiespa3}),(\ref{elemenpos3}) and (\ref{elemeneg3}) one has

\begin{equation}\label{estimafinal3}
diam(W_\delta\cap B_{Y^\circ})\geq \Vert u^*-v^*\Vert\geq \Vert x^*-y^*\Vert-\Vert x^*-u^*\Vert-\Vert y^*-v^*\Vert\\
\geq\end{equation}
\begin{equation}
 2-\frac{\delta}{16}-\frac{\delta}{8}-\frac{\delta}{8}>2-\delta.
\end{equation}

As $\delta\in\mathbb R^+$ was arbitrary we deduce that $diam(W\cap B_{Y^\circ})=2$, as wanted.

{\bf iii)} Assume that $X$ has an octahedral norm and $Y$ is reflexive. We know that, in order to prove that $X/Y$ has octahedral norm, we have to prove that $(X/Y)^*=Y^\circ$ has the $w^*$-strong diameter two property.

To this aim consider $C:=\sum_{i=1}^n \lambda_i S(B_{Y^\circ},x_i,\alpha)$ a convex combination of slices in $B_{Y^\circ}$ and let prove that $diam(C)=2$. 

Pick an arbitrary $\delta\in\mathbb R^+$. Define $S_i:=S(B_{X^*},x_i,\alpha)$ for each $i\in\{1,\ldots, n\}$.

Let $\pi:X^*\longrightarrow X^*/Y^\circ=Y^*$ the quotient map and define $A_i:=\pi(S_i)$.

Now by Phelps's theorem, for each  $i\in\{1,\ldots, n\}$, there exists a convex combination of strongly exposed points of $\overline{A_i}$, say $\sum_{j=1}^{n_i}\mu_j^i a_j^i$, such that

\begin{equation}\label{normafuertexp2}
\left\Vert \sum_{j=1}^{n_i}\mu_j^i a_j^i\right\Vert<\frac{\delta}{16}.
\end{equation}

For every $i\in\{1,\ldots,n\}, j\in\{1,\ldots, n_i\}$ let $(a_j^i)^*\in S_{X/Y}$ a functional which strongly exposes to  $a_j^i$. Given $i\in\{1,\ldots, n\}, j\in\{1,\ldots, n_i\}$ consider $\eta_j^i\in\mathbb R^+$ satisfying

\begin{equation}\label{diaslicescociente2}
diam(S(B_{X^*/Y^\circ},(a_j^i)^*,\eta_j^i)\cap \overline{A_i})<\frac{\delta}{16}.
\end{equation}

Note that $S(B_{X^*/Y^\circ},(a_j^i)^*,\eta_j^i)$ can be seen as a $w^*$-slice because of reflexivity of $Y^*$.

Moreover it is clear that for $i\in\{1,\ldots, n\}$ and $j\in\{1,\ldots, n_i\}$ one has 

$$S(B_{X^*/Y^\circ},(a_j^i)^*,\eta_j^i)\cap A_i\neq \emptyset\Rightarrow S(B_{X^*},\pi^*((a_j^i)^*),\eta_j^i)\cap S(B_{X^*},x_i,\varepsilon)\neq \emptyset.$$

Now $\sum_{i=1}^n \lambda_i \sum_{j=1}^{n_i}\mu_j^i S(B_{X^*},\pi^*((a_j^i)^*),\eta_i)\cap S(B_{X^*},x_i,\varepsilon)$ is a convex combination of non-empty relatively weakly-star open subsets of $B_{X^*}$. Hence, again by Bourgain lemma, it has diameter two. Thus we can find, for each $i\in\{1,\ldots, n\}, j\in\{1,\ldots, n_i\}$, elements $x_j^i, z_j^i\in S(B_{X^*},\pi^*((a_j^i)^*),\eta_j^i)\cap S(B_{X^*},x_i,\varepsilon)$ such that

\begin{equation}\label{estimabinuevos}
\left\Vert \sum_{i=1}^n \lambda_i \sum_{j=1}^{n_i}\mu_j^i(x_j^i-z_j^i)\right\Vert>2-\frac{\delta}{16}.
\end{equation}

Now, following the proof of ii), for each $i\in\{1,\ldots, n\}$ we can find $y_i,y_i'\in S(B_{Y^\circ},x_i,\alpha)$ such that

\begin{equation}\label{elemcercanos}
\begin{array}{ccc}
\left\Vert y_i-\sum_{j=1}^{n_i}\mu_j^i x_j^i\right\Vert<\frac{\delta}{8} & \mbox{and} & \left\Vert y_i'-\sum_{j=1}^{n_i}\mu_j^i
z_j^i\right\Vert<\frac{\delta}{8}.
\end{array}
\end{equation}

So

\begin{equation}\label{perteneslice}
y_i,y_i'\in S\left(B_{Y^\circ}, x_i,\alpha+\frac{\delta}{8}\right),
\end{equation}

and

$$
diam\left(\sum_{i=1}^n \lambda_i S\left(B_Y,y_i^*,\varepsilon+\frac{\delta}{8}\right)\right)\geq \left\Vert \sum_{i=1}^n \lambda_i a_i-\sum_{i=1}^n \lambda_i b_i\right\Vert\geq$$

$$ \geq \left\Vert \sum_{i=1}^n \lambda_i \sum_{j=1}^{n_i}\mu_j^i x_j^i-\sum_{i=1}^n \lambda_i \sum_{j=1}^{n_i}\mu_j^i z_j^i\right\Vert-
\left\Vert a_i-\sum_{j=1}^{n_i}\mu_j^ix_j^i\right
\Vert-\left\Vert b_i-\sum_{j=1}^{n_i}\mu_j^iz_j^i\right
\Vert>$$ $$2-\frac{3\delta}{4}.$$

As $\delta\in\mathbb R^+$ was arbitrary we conclude that $diam(C)=2$.

Hence $X/Y$ has octahedral norm in view of arbitrariness of $C$, so we are done.\end{proof}

\end{document}